\documentclass[a4, 12pt,leqno]{amsart}
\usepackage{lmodern}
\usepackage{bbm}

\usepackage{a4wide}
\usepackage{esint}

\usepackage[active]{srcltx}
\usepackage[utf8]{inputenc}
\usepackage[T1]{fontenc}
\usepackage{lipsum}
\usepackage{upgreek}
\usepackage{listings}
\usepackage{wasysym}
\DeclareMathAlphabet{\mathdutchcal}{U}{dutchcal}{m}{n}
\SetMathAlphabet{\mathdutchcal}{bold}{U}{dutchcal}{b}{n}
\DeclareMathAlphabet{\mathdutchbcal}{U}{dutchcal}{b}{n}
\DeclareSymbolFont{myletters}{OML}{ztmcm}{m}{it}
\DeclareMathSymbol{\nicelambda}{\mathord}{myletters}{"15}
\usepackage{amsmath,amssymb,amscd,amsthm,graphicx,enumerate}
\usepackage{mathrsfs}
\usepackage{amssymb,amsthm,a4wide}
\usepackage{nicefrac}
\newcounter{example}
\newenvironment{example}[1]{\refstepcounter{example}\par\medskip
	\noindent \textsc{\small Example~\theexample. #1} \rmfamily\hspace{-2pt}}{\medskip}
\usepackage{pifont}

\usepackage{color}
\usepackage[usenames,dvipsnames]{xcolor}

\definecolor{ultrablue}{rgb}{0.0,0.0, 1}
\definecolor{jigari}{rgb}{0.39,0.0, 0.0}
\usepackage{hyperref}
\hypersetup{
	linktoc=page,
	colorlinks,
	linkcolor={jigari},
	citecolor={jigari},
	urlcolor={jigari}
}
\hypersetup{breaklinks=true}
\allowdisplaybreaks

\usepackage{tikz}
\usetikzlibrary{matrix}
\usepackage[all]{xy}
\usepackage{subfig}

\makeatletter 
\def\l@subsection{\@tocline{2}{0pt}{3pc}{6pc}{}}
\makeatother

\usepackage{enumerate}
\usepackage{nicematrix}
\usepackage{soul}
\usepackage{graphicx}

\usepackage{tikz}
\usepackage{tikzpagenodes}
\usepackage{scrlayer-scrpage}

\usepackage{tikz}
\usetikzlibrary{calc,spy}
\usetikzlibrary{shapes.geometric, arrows}
\usepackage{pgfplots}
\usetikzlibrary{intersections, pgfplots.fillbetween}

\makeatletter
\@namedef{subjclassname@2020}{%
	\textup{2020} Mathematics Subject Classification}
\makeatother

\makeatletter 
\def\l@subsection{\@tocline{2}{0pt}{3pc}{6pc}{}}
 \makeatother

\usepackage[a4paper,bindingoffset=0in,%
left=2cm,right=2cm,top=2.4 cm,bottom=2.5 cm,%
footskip=1mm]{geometry}

\theoremstyle{plain}

\newtheorem{theorem}{Theorem}

\newtheorem*{theorem*}{Main Theorem}
\newtheorem*{corollary*}{Corollary}

\newtheorem{lemma}{Lemma}[section]
\newtheorem{proposition}[lemma]{Proposition}
\newtheorem{corollary}{Corollary}

\theoremstyle{definition}
\newtheorem{definition}[lemma]{Definition}
\newtheorem{remark}[lemma]{Remark}

\theoremstyle{remark}

\numberwithin{equation}{section}

\newcommand{\R}{\mathbb R}

\newcommand{\B}{\mathcal{B}}

\DeclareMathOperator{\diam}{diam}		

\newcommand{\ident}{\raisebox{0pt}{\scalebox{1.1}{$\mathbbm{1}$}}\hspace{-1pt}}

\usepackage[]{stackengine}

\newlength\correct
\newcommand*{\DashedArrow}[1][]{\mathbin{\tikz [baseline=-0.25ex,-latex, dashed,#1] \draw [#1] (0pt,0.5ex) -- (1.3em,0.5ex);}}%

\begin{document}

\title[\tiny A generalization of the inverse mapping theorem in infinite dimensions]{\small  A generalization of the inverse mapping theorem \\ \small in infinite dimensions} 

\author[Sajjad Lakzian]{Sajjad Lakzian}
\address{ -- Sajjad Lakzian \newline \phantom{s} Department of Mathematical Sciences\newline \phantom{s} Isfahan University of  Technology (IUT) \newline \phantom{s} Isfahan 8415683111, Iran}

\email{\href{mailto:slakzian@iut.ac.ir}{slakzian@iut.ac.ir}}


%
%


\maketitle

\begin{abstract}
\textsl{We present a generalization of the inverse mapping theorem, where variations of a weaker non-expansiveness property (referred to as property ${\sf A}$) replace the key $\mathsf{C}^1$ condition. We also obtain inverse mapping theorems that can be applied to non-smooth maps. Also as a by-product of the generalized inverse mapping theorem, we prove generalizations of the implicit function theorem and existence and uniqueness theorem of abstract PDE systems as well. 
}
\end{abstract}
\date{\today}

\section{Introduction and Main Results}
\par Let $\left(\B, \|\cdot\|\right)$ be a Banach space. The classical inverse mapping theorem (IMT) states that a differentiable $\mathsf{C}^1$ map $f:U\subset \B \to \B$ (where $U$ is an open domain) with invertible (Frech\'et) derivatives everywhere, is a local diffeomorphism. It is a folklore fact that the optimal regularity requirements for the IMT to hold, must strictly lie somewhere between mere differentiability and being $\mathsf{C}^1$. 
\par The continuity of the Fr\'echet derivative is a key assumption that can be removed when $\B$ is finite dimensional; this is attributed to \v{C}ernavski\u{\i}~\cite{Cer} and is known as the everywhere differentiable IMT. Alternative proofs were presented in~\cite{JSR, Vais}. All the proofs in the aforementioned literature heavily rely on the local compactness of (a finite dimensional) $\B$ which characteristically fails in infinite dimensions. Yet as we will see, still local weak compactness can be utilized to get a generalized IMT.
\par There is an extensive literature concerning generalizations of IMT ``beyond the  $\mathsf{C}^1$ realm in infinite dimensions'' and/or ``beyond the everywhere differentiable with non-degenrate derivative in finite dimensions''. Three major approaches are to either
\begin{itemize}
	\item \textsl{assume the non-degeneracy of stronger notions of derivatives}: examples in this direction are abundant. We will mention a few important ones. Clarke's IMT in finite dimensions that applies to Lipschitz maps and only requires the non-degeneracy of Clarke's generalized derivatives~\cite{Clarke}. In \cite{Feckan}, an IMT is proven that works for continuous self maps of reflexive Banach spaces yet it requires the non-degeneracy of a set-valued derivative. Using strong derivatives (or strict derivative), there is a host of IMTs for set-valued maps expressed in terms of strong and metric regularity including Bartle-Graves and Lyusternik-Graves theorems; see~\cite{Dontchev}.
	\item \textsl{to consider weak derivatives but imposing stronger non-degeneracy conditions}: This approach works best in finite dimensions and usually involves integrability assumptions on the distortion of maps. For example local inversion of Sobolev maps falls in this category. See~e.g.~\cite{FG,HK} and the references therein.  
	\item \textsl{to impose appropriate contraction assumptions that are weaker than being $\mathsf{C}^1$}: generalizations that utilize this approach, typically require imposing some sort of a ``contraction'' assumption that will allow the use of the Banach fixed point theorem. For example, see~\cite{Howard} for an overview of the proof of a standard version of IMT for Lipschitz maps with this approach.
\end{itemize}
\par As previously mentioned, infinite dimensional Banach spaces are not locally compact but a large class of them consists of those that are locally weakly compact. Our aim in these notes, is to take the third approach by imposing a (quasi)-non-expansiveness property (unlike the previous works in which strict contraction is required) in combination with weak local compactness. 
\subsection{Property A}
\hfill\\
\noindent\underline{\small \emph{\textbf{Notation:}}} Throughout, $\B_i$s denote Banach spaces. Unless otherwise stated, $f: U \subset \B_1 \to \B_2$ will denote a map from an open subset of $\B_1$ to $\B_2$. The domain of $f$ will be denoted by $\mathsf{Dom}(f)$, mainly to emphasize the cases where the domain is not necessarily open.

\par It is easy to see that for a $\mathsf{C}^1$ map $f$ with an invertible derivative, the following two properties hold true:
\begin{itemize}
	\item For every $a$, the function $x \mapsto x - (D_af)^{-1}f(x)$ is $L_a$-Lipschitz for some $L_a<1$ on a ball $B_{r_a}(a)$ for some sufficiently small $r_a$;
	\item By making $r_a$ smaller if necessary, for all $r<r_a$, and all $y\in B_{\nicefrac{r}{2}}(a)$, the map
	\[
	g_{a,y}: x \mapsto x-(D_af)^{-1}(f(x)-y)
	\]
	is a contracting self-map of $B_r(a)$, see Remark~\ref{rem:C1}. The Lipschitz constant $L_a$ can be chosen arbitrarily close to zero. 
\end{itemize}

Note the elementary fact that $B_r(a)$ is convex. Based on these observations, we introduce the following properties (that will replace $\mathsf{C}^1$).
\begin{definition}[Property ${\sf A}$]\label{defn:prop-a}
	$f: \mathsf{Dom}(f) \subset \B_1 \to \B_2$ is said to satisfy strong ${\sf A}$, ${\sf A}$ or weak ${\sf A}$ (respectively) at an interior point $a$ of the domain, if there exist a radius $s_{a,C_a}>0$, a bounded closed convex neighborhood $C_a$ of $a$, and a continuous injective map $A_{a,C_a}: \B_2 \to \B_1$ such that for all $y\in B_{s_{a,C_a}}(f(a))$, the map
	$
	\tilde{f}_{a,y,A_{a,C_a}}: x \mapsto x-A_{a,C_a}(f(x)-y)
	$
	is respectively 
\begin{itemize}
		\item a contractive self map of $C_a$ i.e. a map with a Lipschitz constant less than $1$. 
	 \item a non-expansive self map of $C_a$ i.e. a map with a Lipschitz constant $1$;
		\item either without fixed points or is "quasi-nonexpansive" (see \textsection\thinspace\ref{subsec:qne} for the definition).
\end{itemize}
\end{definition}	
\noindent\underline{\small \emph{\textbf{Convention:}}} At a general point $a$ (not necessarily in the interior) of the domain, the aforementioned properties hold if $f$ admits a local extension that satisfies the said property at $a$. 
\newline\underline{\small \emph{\textbf{Terminology:}}} When the injective maps $A_{a,C_a}$ can be taken to be linear, we refer to the property as (strong, weak) ${\sf LinA}$.

\par Based on the observation preceding to Definition~\ref{defn:prop-a}, the property ${\sf A}$ for a $\mathsf{C}^1$ map holds on all sufficiently small scales (determined locally uniformly in $a$). In the following definition we introduce a generalization of this property.
\begin{definition}
	Suppose $f: U \subset \B_1 \to \B_2$ satisfies (strong, weak) ${\sf A}$ on a set $\mathcal{R}\subset U$. Let $\mathscr{C}$ denote the collection of the interiors of all possible convex subsets $C_a$ for all $a\in \mathcal{R}$ that satisfy the Definition \ref{defn:prop-a}. 
\begin{itemize}
\item We say that $f$ satisfies property (strong, weak) ${\sf A}$ on $\mathcal{R}$ on all scales, if $\mathscr{C}$ forms a basis for the topology of $U$. 
\item We say $f$ satisfies (strong, weak) ${\sf A}$ on $\mathcal{R}$, (locally) uniformly on all scales when there exist (locally) uniform (in $a$) parameters $\alpha_a,\beta_a, \eta_a,\gamma_a >0$ with $\eta_a \le \gamma_a$ and a locally uniform radius $r_a$, such that for all $r\le r_a$, there exists a convex set $C_a \in \mathscr{C}$ with
\[
\mathrm{dist}(a,\partial C_a) \ge \beta_a \diam(C_a), \quad s_{a,C_a} \ge \alpha_a \diam(C_a), \quad \text{and}, \quad  \eta_ar \le \diam(C_a) \le \gamma_ar.
\]
\end{itemize}
\noindent\underline{\small \emph{\textbf{Convention:}}} For general domains (not necessarily open), we say $f$ satisfies (strong, weak) ${\sf A}$ on all scales when it admits local extensions that do so. Similarly one defines (strong, weak) ${\sf A}$ locally uniformly on all scales for general domains. 
\end{definition}

\begin{lemma}[pairing property]\label{lem:pair}
	Suppose $f^1: \B_1 \to \B_2$ satisfies (weak,strong) $\textsf{A}$ and $f^2: \B_3 \to \B_4$ is $\mathsf{C}^1$  then 
\begin{enumerate}
\item the map $(f^1,f^2): \B_1\times \B_3 \to \B_2 \times \B_4$ satisfies (weak, strong)  $\textsf{A}$;
\item if $f^2$ satisfies (weak, strong)  $\textsf{A}$ everywhere on all scales then so does $(f^1,f^2)$. 
\end{enumerate} 
\end{lemma}

\begin{proof}[\footnotesize \textbf{Proof}]

\par Based on Remark~\ref{rem:C1}, $f^2$ satisfies strong $\textsf{A}$ everywhere on all scales. Let $\tau_a = \frac{1}{2}\diam(C_a)$. 

\par Let $a \in \B_1$ and $b \in \B_3$. 
Note that $C_a \times B_{\tau_a}(b)$ is also a bounded and convex set with 
\[
\diam\left( C_a \times B_{\tau_a}(b) \right) = \sqrt{2}\diam\left(C_{a} \right).
\]

\par It is straightforward to see that
\[
\widetilde{(f^1,f^2)}_{(a,b), (y_1,y_2), (A_{a}, (D_bf^2)^{-1}), C_{a}\times B_{\tau_a}(b)} = \left(\widetilde{f^1}_{a,y_1,A_{a,C_{a}}}, \widetilde{f^2}_{a_2,y_2,(D_bf^2)^{-1}} \right),
\] 
for all $(y_1,y_2) \in B_{s}((a,b))$ for sufficiently small $s$ depending on $C_b$ and $\diam(C_b)$ only.  

\par Item (1) follows due to the following elementary facts: 
\begin{itemize}
	\item The Lipschitz constant of $(\tilde{f^1}, \tilde{f}^2)$ is the minimum of the Lipschitz constants of the constituent maps
	\item The map $(\tilde{f^1}, \tilde{f}^2)$ has a fixed point when both $\tilde{f^1}$ and $\tilde{f^2}$ have fixed points.
	\item If $p_1$ and $p_2$ are respective fixed points then
	\begin{align*}
	\|(\tilde{f^1}(x_1), \tilde{f}^2(x_2)) - (p_1,p_2) \| &= \sqrt{\|f^1(x_1)- p_1\|^2 + \|f^2(x_2) - p_2\|^2} \\
	&\le \sqrt{\|x_1- p_1\|^2 + \|x_2 - p_2\|^2}\\
	& = \|(x_1,x_2) - (p_1,p_2)\|
	\end{align*}
\end{itemize}
 
\par Item (2) follows by considering the following facts: 
\begin{itemize}
	\item Cartesian product of two topology bases is a topology basis for the product topology;
	\item $\displaystyle \mathrm{dist}\left((a,b), \partial\left(C_{a}\times B_{\tau_a}(b)\right)\right) \ge \min\left\{\beta_{a}, \beta_b = \frac{1}{2}\right\}\diam(C_a)$
	\item $\displaystyle s_{(a,b), C_{a}\times B_{\tau_a}(b)} \ge \min\left\{\alpha_{a}, \frac{1}{2} \right\} \diam{C_a}$
\item $\displaystyle \frac{1}{\sqrt{2}} \min\left\{\eta_a, 1\right\}r\le \diam\left( C_a \times B_{\tau_a}(b)  \right) \le \sqrt{2}\max\left\{\gamma_a, 1\right\}r$
\end{itemize}   
Hence one can take
\begin{align*}
\beta_{(a,b)} &= \frac{1}{\sqrt{2}}\min\{\beta_{a}, \beta_b = \frac{1}{2}\}, \quad \alpha_{(a,b)} = \frac{1}{\sqrt{2}} \min\left\{\alpha_{a}, \frac{1}{2} \right\},\\
 \eta_{(a,b)} &= \frac{1}{\sqrt{2}} \min\left\{\eta_a, 1\right\}, \quad\quad\quad \;\hspace{1pt}\gamma_{(a,b)} = \sqrt{2}\max\left\{\gamma_a, 1\right\}.
\end{align*}
\end{proof}

\subsection*{Relations with other regularity conditions}
Note that we do not require $f$ to be continuous, however, (strong) $\textsf{A}$ automatically implies continuity since for every $x_0$, continuity of 
\[
\tilde{f}_{a,y,A_{a,C_a}}
\] 
and $A_{a,C_a}$ imply
\[
A_{a,C_a}(f(x_0^+) - y) = A_{a,C_a} (f(x_-) - y)
\]
and by injectivity of $A_{a,C_a}$, we deduce $f({x_0}^+) = f({x_0}_-)$. 

\par Next let us note that the property weak $\textsf{A}$ does not require $f$ to even be continuous or have continuous inverse (let alone Lipschitz or metric regular) as the following example illustrates.
\begin{example}\phantom{}\hspace{-8pt} Let $c>0$. Let $h_{a}:\R \to \R$ denote the function
	\[
	h_{a}(x) := \begin{cases}  
		x-a - \frac{c}{2} & x \le a\\
		x-a + \frac{c}{2} & x > a 
		\end{cases}
	\] 
Let $A(x) = x$. Then
\[
\tilde{h}_y: x \mapsto x - A(f(x) - y)
\] 
is given by
\[
\tilde{h}_y(x) = \begin{cases} 
	a+\frac{c}{2}+y   & x \le a\\
	a-\frac{c}{2} + y & x > a
	\end{cases}
\]
hence for sufficiently small $y$ ($|y|< \frac{c}{2}$), $\tilde{h}_y: [a-c,a+c] \to [a-c,a+c]$ does not have fixed points. Hence $h_a$ satisfies a weak $\textsf{A}$ at $a$.
\label{ex:dis-cont} 
\end{example} 

Note that in Example~\ref{ex:dis-cont}, $h_a$ is invertible with a discontinuous inverse thus $h$ also fails to be metric regular at $a$.

The following example illustrates that a non-Lipschitz function can satisfy strong $\text{A}$.

\par Understanding a function that satisfies property (weak, strong) $\textsf{A}$ on all scales and its regularity is a challenging question that we do not have a definite answer for. As far as the inverse function theorem in concerned, ``$\mathsf{C}^1$ + invertible derivative'' is stronger than property $\text{A}$ as in argues in the following Remark.

\begin{remark}\label{rem:C1}
	A $\mathsf{C}^1$ map $f$ with an invertible derivative satisfies strong ${\sf LinA}$, locally uniformly on all scales. To see this, we first note that by differentiability, it holds
	\begin{align}\label{eq:1}
	x-a - (D_af)^{-1}(f(x)-y) = o(\|x-a\|) + (D_af)^{-1}(y - f(a)), \quad \text{as} \quad x \to 0.
	\end{align}
Thus, there exists an $r_a>0$ such that for all $r\le r_a$, and $y \in B_{\nicefrac{r}{2}}(f(a))$, the map
	$
	\tilde{f}_{a,y,(Df)^{-1}} 
	$
	is a self map of $B_r(a)$ and $\|D\tilde{f}_{a,y,(Df)^{-1}}\|\le \tau_a 
	<1$ for some $\tau_a \in (0,1)$. This means $\tilde{f}_{a,y,(Df)^{-1}}$ is a contractive map. 
\par It is a well-known fact that continuous differentiability is equivalent to locally uniform differentiability (i.e. ~\eqref{eq:1} holds locally uniformly in $a$); e.g. see ~\cite[Page 35]{yamamuro}. This implies that we can choose $r_a$ locally uniformly in $a$. 
\par Taking the convex sets $C_a$ to be $B_r(a)$ for all sufficiently small $r$ and letting $\beta_a = \alpha_a = \frac{1}{2}$ and $\eta_a = \gamma_a = 1$, the claim follows. 
\par Therefore, property ${\sf A}$ locally uniformly on all scales is weaker than $\mathsf{C}^1$ with invertible derivative. For Lipschitz maps, the satisfaction of $\textsf{A}$ on a locally uniform scale is again an intricate issue that will not be explored in these notes. 
\end{remark}

\begin{remark}
\par Note that the first part of the argument in Remark~\ref{rem:C1} only relies on the differentiability at $a$. Therefore, the same argument shows that if $f$ is a Lipschitz map with a	 non-degenerate Clarke's generalized Jacobian, then it satisfies property $\text{A}$ at least at all of its differentiability points. By Rademacher's theorem, this means that 	at almost every point when the Banach spaces $\B_1$ and $\B_2$ are finite dimensional. In the infinite dimensions, the latter implies that $f$ satisfies $\text{A}$ on a dense open subset of $\B_1$ provided $\B_1$ is norm differentiable or Asplund; this is indeed according to the infinite dimensional generalization of the Rademacher's theorem; see the seminal work~\cite{Preiss}.
\end{remark}
\subsection{Main Results}
\hfill\newline
\noindent\underline{\small \emph{\textbf{Convention:}}} $f$ is said to be differentiable at a general point $a$ (not necessarily an interior point) of the domain, when $f$ admits a local extension at $a$ that is differentiable. We say $f$ has an invertible derivative at $a$, when at least one of the said local extensions around $a$ does so. 
\begin{definition}[the fixed point property (${\sf FPP}$)]
	A Banach space $\B$ is said to have the fixed point property if any continuous non-expansive self mappings $f:C \to C$, for bounded closed convex subsets $C$, has at least one fixed point.
\end{definition}
\begin{theorem}[Generalized IMT]\label{thrm:main-1}
Let $f: \mathsf{Dom}(f) \subset \B_1 \to \B_2$ be an everywhere differentiable map with invertible derivative. Then $f$ is a local diffeomorphism provided any of the following items hold
	\begin{enumerate}
		\item if $f$ satisfies strong ${\sf A}$ everywhere on all scales;
		\item If $\B_1$ is locally weakly compact and $f$ is weak-weak continuous and satisfies ${\sf A}$ everywhere on all scales;
		\item If $\B_1$ is locally weakly compact with ${\sf FPP}$ and $f$ satisfies ${\sf A}$ everywhere on all scales; 
		\item $\B_1$ is strictly convex with ${\sf FPP}$ and $f$ satisfies weak ${\sf A}$ everywhere on all scales. 
	\end{enumerate} 
\end{theorem}
\begin{remark}
A large class of Banach spaces $\B_1$ that the last theorem applies to, are $k$-uniformly rotund spaces for $k \in \mathbb{N}$. These spaces enjoy ${\sf FPP}$, see~\textsection\thinspace\ref{subsec:qne}. These are also strictly convex and reflexive therefore, they are locally weakly compact, see~\textsection\thinspace\ref{subsec:wt}. 
\end{remark}
\par The following two corollaries are generalizations of Hadamard type global inverse mapping theorems, see~\cite{Gutu} for a survey on global inverse mapping theorems. 
\begin{corollary}\label{corol:main-1}
	If $f: \mathsf{Dom}(f) \subset \B_1 \to \B_2$ is as in the Main Theorem, and is furthermore either
	\begin{enumerate}
		\item proper
		\item[]or,
		\item with compact domain
	\end{enumerate}
	then, it is a finite-sheeted covering map. 
	\par In particular if in addition, $\mathrm{image}(f)$ is simply connected, then $f$ is a (global) diffeomorphism. 
\end{corollary}
\begin{corollary}\label{corol:main-2}
	If $f: \B_1 \to \B_2$ satisfies either of the the hypotheses of the Main Theorem and
	\[
	\int_0^\infty \inf_{\|f(x)\|\le s} \|D_xf\| \; ds = \infty,
	\]
	holds (this is referred to as the Hadamard-Levy criterion), then $f$ is a (global) diffeomorphism. In particular, this holds when $\|D_xf\|$ is uniformly bounded away from zero.  
\end{corollary}
\begin{corollary}\label{corol:main-3}
Suppose $f: \mathsf{Dom}(f) \subset \B_1 \to \B_2$ is a locally injective continuous map. Then $f$ is a local homeomorphism provided either of the following happens
\begin{enumerate}
	\item $f$ satisfies strong ${\sf A}$ on a dense set $\mathcal{R}$ locally uniformly on all scales; 
	\item If $\B_1$ is locally weakly compact and $f$ is weak-weak continuous and satisfies ${\sf A}$ on a dense set $\mathcal{R}$ locally uniformly on all scales; 
	\item $\B_1$ satisfies ${\sf FPP}$ and $f$ satisfies ${\sf A}$ on a dense set $\mathcal{R}$, locally uniformly on all scales.
	\item  $\B_1$ is strictly convex with ${\sf FPP}$ and $f$ satisfies weak ${\sf A}$ on a dense set $\mathcal{R}$, locally uniformly on all scales.
\end{enumerate}
\end{corollary}

\par In what follows we wish to provide criteria to ensure certain maps that are not everywhere differentiable, are still almost diffeomorphism in a certain suitable sense. Ideally one wishes to obtain an inverse mapping theorem that works for Lipschitz maps. But our analysis does not yet work in such generality. However this can be done for certain maps that are in a sense more regular than Lipschitz maps but less than everywhere differentiable with an invertible derivative. 

First we need to introduce certain notion of large sets in terms of a convexity like property that works well with our tools. 
\begin{definition}
	A subset $S$ of $\B_1$ is said to be segmentally dense if for any line segment $L$ with endpoint in $S$, we have $S \cap L$ is dense in $L$. 
\end{definition}
\noindent\underline{\small \emph{\textbf{Notation:}}} We denote the set of points at which $f$ is differentiable by $\mathsf{FD}$ and the points where derivative is invertible (non-degenerate) by $\mathsf{FD}^{\pm}$.

\par The following notions are large scale analogues of a map having a non-degenerate derivative. 
\begin{definition}
	\hfill
	\begin{itemize}
		\item The map $f$ is said to have \emph{weakly non-degenerate directional derivatives}, if the directional derivatives of $f$ along any line and in the direction of the line, are nonzero on a dense subset of that line;
		\item The map $f$ is said to be segmentally non-degenerate when $\mathsf{FD}^{\pm}$ is a segmentally dense set. 
	\end{itemize}
\end{definition}
The geometric intuition is that one should think of maps with properties \emph{segmentally non-degenerate} and/or \emph{having weakly non-degenerate directional derivatives} as maps that do not collapse non-trivial line segments to zero length. 

\begin{definition}
	The map $f: \mathsf{Dom}(f)\subset \B_1 \to \B_2$ is said to be an almost diffeomorphism when it is a homeomorphism onto its image and $f$ and $f^{-1}$ are differentiable on dense sets. 
\end{definition}

\begin{corollary}\label{corol:main-4}
	Suppose $f: \mathsf{Dom}(f) \subset \B_1 \to \B_2$ and $\B_1$ satisfy any of the items (1)--(4) as in Corollary~\ref{corol:main-3}.
\begin{enumerate}
	\item If $f$ has almost non-degenerate directional derivatives 
	then, it is a local homeomorphism.
	\item If $f$ is segmentally non-degenerate, then it is a local almost diffeomorphism. 
\end{enumerate}
\end{corollary}

Another important consequence is the generalized implicit function theorem. 

\begin{theorem}[Generalized implicit function theorem]\label{thrm:main-2}
	Suppose $\B_1$ and $\B_2$ are two Banach spaces and $g:W \subset \B_1 \times \B_2 \to \B_2$ is a function which we will write it as $g(t,x)$. Suppose $g$ is everywhere differentiable and with invertible derivative $D_x g(t,x)$ with respect to the variable $x$. Assume $g$ and the domain $\B_1\times \B_2$ satisfy either of the hypotheses of Theorem~\ref{thrm:main-1}. Then, every non-empty level set $g^{-1}(c)$ can be locally uniquely written, around $(a,b) \in g^{-1}(c)$ as a graph of a differentiable function $h: U_a \subset \B_1 \to \B_2$. 
\end{theorem}

\subsection{Road map and Organization}
In \textsection\thinspace\ref{sec:prelim}, we provide some necessary background material needed for better readability of these notes. 

In \textsection\thinspace\ref{sec:Proofs}, we provide the detailed proofs of the main and auxiliary results. In particular, the road map of the generalized IMT is as follows:

\begin{enumerate}
	\item Variants of property $\textsf{A}$ along with suitable hypotheses on the domain Banach space in each case, is utilized to prove the existence of fixed points of the maps $\tilde{f}$; this in turn results in local surjectivity;
	\item Differentiability along with an invertible derivative is argued to provide discreteness of the map $f$ i.e. the fact it has discrete fibers;
	\item Invoking the convexity hypotheses, we show local injectivity, thus establishing $f$ is a local homeomorphism;
	\item Revisiting differentiability with an invertible derivative, we show local diffeomorphism. 
\end{enumerate}  

\par In order to demonstrate the significance of these results, we have included an application to abstract ODE theory, of the generalized theorems we have established in \textsection\thinspace\ref{sec:application} where we show existence and uniqueness for a certain type of ODE system without continuity or Lipschitz condition on the RHS of the equation. 

\addtocontents{toc}{\protect\setcounter{tocdepth}{-1}}
\section*{\small \bf  Acknowledgments}
\addtocontents{toc}{\protect\setcounter{tocdepth}{1}}

\vspace{-175pt}

\begin{minipage}[c][7cm][b]{0.87\textwidth}
\renewcommand{\labelitemi}{\raisebox{1pt}{\scalebox{0.6}{\ding{169}}}}
	\begin{itemize}
		\item \textit{The author would like to express his gratitude to the anonymous referees whose valuable comments helped the author improve this article}. 
	\end{itemize}
\end{minipage}
\normalsize

\section{Preliminaries}\label{sec:prelim}
\subsection{Weak topology and local weak compactness}\label{subsec:wt}
The weak topology on a Banach space is one that topologizes weak convergence of sequences. In this paper, a weakly continuous map $f: \mathsf{Dom}(f)\subset \B_1 \to \B_2$ between Banach spaces, refers to a map that is weak-weak continuous. 
\par A well-known consequence of Banach-Alaoglu's theorem (known as Kakukati's theorem) is that reflexive Banach spaces are locally weakly compact; see~\cite[Theorem V.4.2]{Conw}
\par By the Schauder-Tychonoff fixed point theorem, when $\B$ is locally weakly compact, any weakly continuous self map $f: C \to C$ admits a fixed point when $C$ is a bounded closed convex set. 
\subsection{Convexity of a Banach space}
\hfill\\
\noindent\underline{\small \emph{\textbf{Terminology:}}} In some older literature it is customary to refer to convexity (mainly in the context of uniform ones) as rotundity.
\par Characteristically, in any normed linear space, any ball is convex. Many important fixed point results in nonlinear functional analysis often requires stronger notions of convexity. 
\par Recall the modulus of convexity of a Banach space $\B$ is the non-decreasing function $\mathsf{m_c}: [0,2] \to [0,1]$ given by~\cite{Opial}
\[
\mathsf{m_c}(\epsilon) := \inf_{\substack{\|x\|= \|y\|=1\\
\|x-y\|\ge \epsilon}} \left( 1 - \frac{\|x+y\|}{2}  \right).
\]
It is known that $\mathsf{m_c}$ is convex; see~\cite{Goeb}.
\par The characteristic of convexity is defined as~\cite{Goeb}
\[
\epsilon_0 := \sup \left\{\epsilon \; \;\text{\textbrokenbar}\;\; \mathsf{m_c}(\epsilon) = 0 \right\}.
\]

\subsubsection{Strictly convex space}
$\B$ is said to be strictly convex if the balls are strictly convex equivalently if $\mathsf{m_c}(2) = 1$; see \cite{Clarc}.

\subsubsection{Uniformly convex space}
Uniform convexity is stronger than strict convexity.
\begin{definition}[uniform convexity]
$\B$ is said to be  uniformly convex if 
\[
\forall \epsilon>0, \; \exists \delta >0 , \;\; s.t. \;\; \|x\|=\|y\ = 1 \;\text{and} \;|\|x-y\|\ge \epsilon \Longrightarrow \|\frac{x+y}{2}\| \le  1 - \delta 
\]
equivalently $\B$ is uniformly convex if and only if $\epsilon_0 = 0$. Uniformly convex spaces are reflexive (Milman-Pettis theorem).
 
\end{definition}

\subsubsection{$k$-uniformly convex spaces}
\par The $k$-uniformly convex (originally called $k$-uniformly rotund or $k$-UR) spaces were introduced in~\cite{Sul} as generalizations of uniformly convex spaces. These spaces are reflexive (as uniformly convex Banach spaces are).
\begin{definition}[k-uniformly rotund (convex) \cite{Sul}]
	$\B$ is called $k$-UR if 
	\[
	\forall \epsilon >0, \;\; \exists \delta>0 \;\; s.t. \;\; \|x_1\| = \cdots = \|x_{k+1}\|=1 \; \text{and} \;  V(x_1,x_2, \cdots,x_{k+1}) \Longrightarrow \left\|\frac{\sum_{i=1}^{k+1} x_i}{k+1}\right\| \le 1 - \delta
	\]
where $V(x_1, \cdots, x_{k+1})$ is the $k$-dim volume given by
\[
V(x_1, \cdots, x_{k+1}) := \sup_{\substack{f_i \in \B^* \\ \|f_i\| \le 1}} \det \begin{pmatrix} 1 & 1 & \cdots & 1 \\ f_1(x_1) & f_1(x_2) & \cdots & f_1(x_{k+1}) \\ \vdots & \ddots & \ddots & \vdots\\ f_k(x_1) &  f_k(x_2)& \cdots& f_k(x_{k+1}) \end{pmatrix}.
\]
\end{definition}

\subsection{Fixed points of (quasi)-nonexpansive maps}\label{subsec:qne}
\subsubsection{Fixed points of non-expansive maps}
\par A nonexpansive map is simply one with Lipschitz constant 1 i.e. it does not increase distances. 
\par According to the Browder-G\"ohde-Kirk fixed point theorem~\cite{Edel}, a non-expanding self map on a bounded closed convex set $C$ in a uniformly convex $\B$, admits a fixed point. It is also well-known that if $\B$ is merely strictly convex, then $f$ has a convex set of fixed points (possibly empty).   
\par A simple argument combined with the convexity results obtained in~\cite{Bruck} for nonexpansive maps shows that, when $C$ is a bounded closed convex subset of $\B$ and $\B$ satisfies ${\sf FPP}$, then the set of fixed points of a non-expansive self map $f$ on $C$ is a closed convex subset of $C$, see Proposition~\ref{prop:key-2}.
\par  It was proven in~\cite{MY} that $k$-uniformly rotund spaces satisfy the ${\sf FPP}$.
\subsubsection{Fixed points of quasi-nonexpansive maps}\label{subsubsec:fp-qn}
A (not necessarily continuous) $f:C \to C$ ($C$ closed and convex) is called quasi-nonexpansive if it has at least one fixed point and distances to the fixed point are not increased under it i.e. if for all such fixed points $p \in C$, it satisfies 
\[
\|f(x) - f(p)\|\le \|x-p\|, \quad \forall x\in C.
\]
This was introduced in~\cite{DM}. 
\par In~\cite{Dotson}, it has been shown that if $\B$ is strictly convex, then the set of fixed points of a quasi-nonexpansive map is a nonempty closed convex set on which $f$ is continuous. 
\section{Proofs}\label{sec:Proofs}
\par Note that openness, local surjectivity and local injectivity are inherited by maps restricted to smaller domains and targets (equipped with induced topologies). Since all the arguments are local and with the conventions previously made, in the proof of the fact that $f$ is a local homeomorphism, we can without loss of generality, assume that the domain of $f$ is open. So we will proceed with this assumption. 
\par Let us note the simple yet key fact that, due to the injectivity of the maps $A_{a,C_a}$, one deduces
\[
f^{-1}(y) \cap C_a = \text{the set of fixed points of $\tilde{f}_{a,y,A_{a,C_a}}$}.
\]
\emph{When all the variables can be inferred form the context, with a slight abuse of notation, we will simply write $\tilde{f}_{a,y}$ and suppress the other variables. Similarly we will write $s_a$ and $A_a$}.
\subsection{Key Lemmas}
\begin{lemma}\label{lemma:key-1}
	If $f: U \subset \B_1 \to \B_2$ satisfies either of the hypotheses of the Main Theorem, then it is open and locally surjective.
\end{lemma}
\begin{proof}[\footnotesize \textbf{Proof}]
For local surjectivity, it suffices to show that for every $a$, and $y\in B_{s_a}(f(a))$ the map $\tilde{f}_{a,y}$ has at least one fixed point. Indeed, this implies $B_{s_a}(f(a)) \subset f(C_a)$. The ${\sf FPP}$ assumption in hypotheses 3 and 4, guarantees the existence of such a fixed point. In hypothesis 1, the Banach fixed point theorem can be applied. In hypothesis 2, the Leray-Schauder fixed point theorem can be invoked. Therefore in all cases, local surjectivity follows. 
\par For openness, we need to show that the image of every open ball around some $a$ contains an open ball around $f(a)$. Consider the open ball $B_{r}(a)$. Since we have assumed property ${\sf A}$ on all scales, there exists a closed bounded convex set $C_a \in \mathscr{C}$ with $C_a \subset B_r(a)$. By the first part, $f(C_a)$ contains $B_{s_a}(f(a))$ which implies that $f(B_r(a))$ contains the ball $B_{s_{a}}(f(a))$. Therefore, $f$ is an open map. 
\end{proof}

The following is by now a standard result. For the definition of normal structure, we will refer the reader to~\cite{BKS}. 
\begin{proposition}\label{prop:key-2}
	Suppose $\B$ has a normal structure (in particular, if it is locally weakly compact (this contains reflexive spaces)) with ${\sf FPP}$. Let $g: C \to C$ be a non-expanding self map of a bounded closed convex set $C$. Then the set of fixed points of $g$ is a non-empty closed convex subset of $C$.
\end{proposition}
\begin{proof}[\footnotesize \textbf{Proof}]
See~\cite{Bruck}.
\end{proof}
\begin{lemma}\label{lemma:2.5}
	If $f$ is differentiable everywhere with an invertible derivative, then it is a discrete map i.e. for every $y$, $f^{-1}(y)$ is a discrete set. 
\end{lemma}
\begin{proof}[\footnotesize \textbf{Proof}]
Since discreteness of a map is preserved under compositions with linear isomorphisms, we can assume, without loss of generality, that $a=0$, $y = f(0)=0$ and $D_0f = \ident$.  
\par The definition of differentiability implies that for each $\varepsilon>0$, there exists $r(\epsilon)>0$ such that
\begin{align}\label{eq:1}
	\|f(x) - x\|\le \varepsilon\|x\|, \quad 0<\|x\|< r(\epsilon),
\end{align}
and thus one deduces $f^{-1}(0)\cap B_r(0) = \left\{0\right\}$. This means that every $x \in f^{-1}(y)$ is an isolated point. Thus the inverse images of points are discrete sets.
\end{proof}
\begin{lemma}\label{lemma:key-3}
	Suppose $f$ is a discrete map and for every $C_a \in \mathscr{C}$ with sufficiently small diameter, and for every fixed $y \in B_{s_a}(f(a))$, we know that
	$
	f^{-1}(y) \cap C_a
	$
	is a non-empty closed convex set. Therefore, $f$ is locally injective at $a$. 
\end{lemma}
\begin{proof}[\footnotesize \textbf{Proof}]
\par Suppose, on the contrary that $f$ is not locally injective at $a$. Take $y\in B_{s_a}(f(a))$ and $z_1, z_2 \in C_a$ with	 $f(z_1)=f(z_2) = y$ where $z_1 \neq z_2$. According to the convexity hypothesis, we have $[z_1,z_2] \subset f^{-1}(y)$ which contradicts the discreteness of $f$. 
\end{proof}
\par The following is well-known, yet the proof is provided for the sake of thoroughness. 
\begin{lemma}\label{lemma:key-4}
	Suppose $f$ (with an open domain) is a local homeomorphism and is differentiable at $a$ with an invertible derivative. Then, the local inverse $f^{-1}$ is differentiable at $f(a)$ and $D_{f(a)}f^{-1} = (D_af)^{-1}$. 
\end{lemma}
\begin{proof}[\footnotesize \textbf{Proof}]
	Without loss of generality, let us assume $a=f(a)= 0$ and $D_0f = \ident$. Then we have
	\[
	\lim_{y=f(x) \to f(a)}\frac{\|f^{-1}(y)\|}{\|y\|}=\lim_{x \to 0} \frac{\|x\|}{\|f(x)\|} = 1.
	\]
	Hence,
	\begin{align}
		\|x - f(x)\| = o(\|f(x)\|), \quad \text{as $\|f(x)\| \to 0$},
	\end{align}	
	and the claim follows. 
\end{proof}
\subsection{Proof of the main results}
\subsubsection{\small \textbf{\textsf{Proof of Theorem~\ref{thrm:main-1}}}}
\par The openness (and consequently the local surjectivity) follows from Lemma~\ref{lemma:key-1}. We only need to argue the local injectivity of $f$; since this implies that $f$ is a local homeomorphism. Thus, by Lemma~\ref{lemma:key-4}, it is a local diffeomorphism. 
\par In \emph{\small Statement (1)}, the local injectivity is guaranteed by the uniqueness of he fixed point in the Banach fixed point theorem. 
\par In \emph{\small Statements (2), (3)}, the local injectivity follows from the discreteness established in Lemma~\ref{lemma:2.5} together with Proposition~\ref{prop:key-2} and Lemma~\ref{lemma:key-3}.
\par In \emph{\small Statement (4)}, the local injectivity follows from Lemma~\ref{lemma:key-3} in combination with the fact that in strictly convex spaces, the set of fixed points of quasi-nonexpansive maps form a convex set, see~\textsection\thinspace\ref{subsubsec:fp-qn}. 
\qed
\subsubsection{\small \textbf{\textsf{Proof of the Corollary~\ref{corol:main-1}}}}
\subsubsection*{\small Statement (1)}
Properness and being a local homeomorphism ensure that the inverse image of each point is a finite set which in turn implies that such a point admits a neighborhood that is evenly covered. Consequently $f$ has the unique lifting property, see the proof of~\cite[Proposition 1.30]{Hatch}. 
\par A local homeomorphism with the unique lifting property with a path connected domain and locally simply connected target, is a covering map; e.g. see~\cite[Section 5.6, Proposition 6]{Doc}. Applying the above argument to path connected components of $U$ completes the proof of the claim.
\subsubsection*{\small Statement (2)}
The claim directly follows from the topology fact that a local homeomorphism from a compact Hausdorff space to a Hausdorff space is a covering map.
\qed
\subsubsection{\small\textbf{\textsf{Proof of the Corollary~\ref{corol:main-2}}}}
This follows from the Main Theorem in combination with the results in~\cite{Play}.
\qed
\subsubsection{\small\textbf{\textsf{Proof of the Corollary~\ref{corol:main-3}}}}
\par Let $x$ be fixed. By the qualifier ``locally uniform on all scales'', there exists a neighborhood $U_x$ such that $\alpha_a, \beta_a >0$ can be taken to be the same constants for all $a \in U_x \cap \mathcal{R}$.  
\par Since $\mathcal{R}$ is dense, we can choose $a \in \mathcal{R}$ arbitrarily close to $x$. Now for sufficiently small $r$, there exists $a\in \mathcal{R}$ and $C_a$ with $ \eta_a r \le \diam(C_a)\le \gamma_ar$ such that $B_{2\gamma_ar}(x)$ contains $C_a$. Consequently by the proof of Lemma~\ref{lemma:key-1}, $f(B_{2\gamma_ar}(x))$ contains $B_{\alpha_a\eta_ar}(f(a))$ which contains $B_{\frac{1}{2}\alpha_a\eta_ar}(f(x))$ provided we had chosen $a$ with $\|f(x)-f(a)\| \le \frac{1}{2}\alpha_a\eta_ar$. 
\par Thus openness (thus also local surjectivity) follows. 
  \qed
\subsubsection{\small\textbf{\textsf{Proof of the Corollary~\ref{corol:main-4}}}}
\par Again local surjectivity and openness follows by a similar argument as in the proof of Corollary~\ref{corol:main-3}. The local injectivity is argued as follows. Suppose $f$ is not locally injective. Thus there exists $a$ and $C_a$ and $y$ such that the fixed points of $\tilde{f}_{a,y}$ form a bounded closed convex set that is not a singleton. In particular, it includes a line segment. The directional derivative of $f$ along any line segment and in the direction of the line segment exists and vanishes. This contradicts both ``non-degenerate directional derivatives'' and ``being segmentally non-degenerate''.
\qed

\subsubsection{\small\textbf{\textsf{Proof of the Theorem~\ref{thrm:main-2}}}}
\par Consider the map
	\[
	\mathcal{F}: \B_1 \times \B_2 \to \B_1\times \B_2,
	\]
	given by
	\[
	\mathcal{F}(t,x) = \left(t, g(t,x)\right).
	\]
	Note
	\[
	D_{(t,x)}\mathcal{F} = \begin{pmatrix} \ident &  0 \\ D_tg & D_xg\end{pmatrix}.
	\]
	
\par Since $D_xg$ is invertible, we deduce $D_{(t,x)}\mathcal{F}$ is invertible as well. 
	
\par According to the key pairing Lemma (Lemma~\ref{lem:pair}), $\mathcal{F}$ satisfies the same hypothesis of the Theorem~\ref{thrm:main-1} as $g$ does. Thus by Theorem~\ref{thrm:main-1}, we deduce $\mathcal{F}$ is a local diffeomorphism. Let $\mathcal{F}^{-1}$ be the ``unique'' local inverse around $(0,c) \in \B_1\times \B_2$ whose image contains a neighborhood of $(0,b)$. 
	
\par It is straightforward to see that 
	\[
	\mathcal{F}^{-1}(t,y) = \left(t, k(t,y)\right),
	\]
	for some $k:\B_1\times\B_2 \to \B_2$ thus $k(t,g(t,x)) = (t,x)$. This means locally $g^{-1}(c)$ is the graph of the unique function $t \mapsto k(t,c)$. 
\qed
\section{An application in abstract differential equation systems}\label{sec:application}

\noindent\underline{\small \emph{\textbf{Notation:}}} For simplicity, in what follows, a function between two Banach spaces means a partial function i.e. a function defined on an open subset and this is signified by using the notation
$\DashedArrow$.\newline
\noindent\underline{\small \emph{\textbf{Convention:}}} We will suppress mentioning the domains of functions. The common sense assumptions when dealing with an pde are in order. 

\par Suppose $\B_1$ and $\B_2$ are two Banach spaces. We use variables $t$ and $x$ for elements of $\B_1$ and $\B_2$ (resp.). 

Let
\[
\mathcal{G}:= \left\{g: \B_1 \times \B_2 \DashedArrow \B_2 \; \;\text{\textbrokenbar}\; \text{$D_xg$ is invertible} \right\}.
\]
We furtheremore assume for each $t$, $g(t,\cdot)$ and the Banach space $\B_2$ satisfy either of the hypotheses in Theorem~\ref{thrm:main-1}. Consider the function space
\[
\mathcal{H}:= \left\{f: \B_1 \times \B_2 \DashedArrow \B_2 \; \;\text{\textbrokenbar}\; f(t,x) := \left[D_xg(t,x)\right]^{-1} D_tg(t,x) \quad \text{for some} \quad g \in \mathcal{G}\right\},
\]
and note that $f\in \mathcal{H}$ is neither necessarily Lipschitz nor even necessarily continuous. 

\par Consider the abstract differential equation (system)
\begin{align}\label{eq:abs-ode}
\begin{cases} u'(t) = f(t,u(t)), \quad f \in \mathcal{H}\\
	u(0) = b\in \B_1
\end{cases}.
\end{align}
The following theorem is a generalization of the standard existence and uniqueness theorem to equations where the RHS is not necessarily Lipschitz or even continuous though belong to $\mathcal{H}$

\begin{theorem}
	The abstract pde system \eqref{eq:abs-ode} has a unique solution locally given by
	\[
	g(t,u(t)) = c = g(0,b).
	\]
\end{theorem}

\begin{proof}[\footnotesize \textbf{Proof}]
\par Let us rewrite the PDE \ref{eq:abs-ode}. From
\[
u'(t) = \left[D_xg(t,u(t))\right]^{-1} D_tg(t,u(t)),
\]
(this is a pde since $t$ comes from a multi-dimensional space)
we deduce
\[
D_xg(t,u(t))\circ u'(t) - D_tg(t,u(t)) = 0.
\]
which by virtue of the chain rule, the latter is just
\[
D_t\left( g(t,u(t))  \right) = 0.
\]
Thus any possible solution lies on a the level set $g^{-1}(c)$. 

\par By the IFT proven in Theorem~\ref{thrm:main-2}, level set $g^{-1}(c)$ around any of its points is locally the graph of a unique function $t \mapsto u(t)$. It is then obvious that $u(t)$ solves the system \eqref{eq:abs-ode} and is unique since by IFT, locally it is uniquely determined. 
\end{proof}

\end{document}